\documentclass[11pt,reqno]{amsart}
\usepackage[foot]{amsaddr}

\usepackage{amsmath} 
\usepackage{amsfonts,amssymb,amsthm,bm}
\usepackage{amsrefs}

\usepackage{enumitem}
\usepackage[utf8]{inputenc}
\usepackage[T1]{fontenc}
\usepackage{stackrel}

\usepackage{fullpage}

\usepackage[dvipsnames]{xcolor}

\usepackage{hyperref}

\def\dd{\mathrm{\,d}} 
\newcommand\abs[1]{\left|#1\right|} 
\newcommand\norm[1]{\lVert#1\rVert} 
\newcommand\normp[2]{\left[#1\right]_{#2}} 

\DeclareMathOperator{\dist}{\mathit{d}} 

\newcommand{\Chi}{\mathcal{X}} 
\DeclareMathOperator*{\essinf}{ess\,inf}
\DeclareMathOperator*{\esssup}{ess\,sup}
\DeclareMathOperator*{\loc}{loc} 

\def\Xint#1{\mathchoice
{\XXint\displaystyle\textstyle{#1}}%
{\XXint\textstyle\scriptstyle{#1}}%
{\XXint\scriptstyle\scriptscriptstyle{#1}}%
{\XXint\scriptscriptstyle\scriptscriptstyle{#1}}%
\!\int}
\def\XXint#1#2#3{\mkern3mu{\setbox0=\hbox{$#1{#2#3}{\int}$ }
\vcenter{\hbox{$#2#3$ }}\kern-.6\wd0}}

\def\dashint{\Xint-}

\theoremstyle{plain}
\newtheorem{theorem}{Theorem}
\newtheorem{propo}[theorem]{Proposition}
\newtheorem{lemma}[theorem]{Lemma}
\newtheorem{coro}[theorem]{Corollary}
\numberwithin{theorem}{section}

\theoremstyle{remark}
\newtheorem{remark}[theorem]{Remark}

\theoremstyle{definition}
\newtheorem{dfn}[theorem]{Definition}

\begin{document}

\title{Limiting conditions of Muckenhoupt and reverse Hölder classes on metric measure spaces}
\date{\today}

\author{Emma-Karoliina Kurki}
\address[]{Aalto University, Department of Mathematics and Systems Analysis, P.O. BOX 11100, FI-00076 Aalto, Finland}
\email{emma-karoliina.kurki@aalto.fi}

\keywords{doubling metric space, annular decay, Muckenhoupt weights, reverse Hölder inequality, natural maximal function}
\subjclass[2010]{42B35, 42B25, 30L99}

\thanks{\emph{Acknowledgements.} I thank Juha Kinnunen for his help and prompts during the preparation of this article.}

\begin{abstract}
The natural maximal and minimal functions commute pointwise with the logarithm on $A_\infty$. We use this observation to characterize the spaces $A_1$ and $RH_\infty$ on metric measure spaces with a doubling measure. As the limiting cases of Muckenhoupt $A_p$ and reverse Hölder classes, respectively, their behavior is remarkably symmetric. On general metric measure spaces, an additional geometric assumption is needed in order to pass between $A_p$ and reverse Hölder descriptions. Finally, we apply the characterization to give simple proofs of several known properties of $A_1$ and $RH_\infty$, including a refined Jones factorization theorem. In addition, we show a boundedness result for the natural maximal function. 
\end{abstract}

\maketitle

\section{Introduction}

The \emph{natural maximal function} $M^\natural$ is a variant of the Hardy--Littlewood maximal function dropping the absolute value signs. As shown by Ou \cites{MR1840094,MR2407089} in $\mathbb{R}^n$, it commutes pointwise with the logarithm on $A_\infty$, and so does its close relative the \emph{natural minimal function} $m^\natural$. Starting from this observation, we describe the structure of the spaces $A_1$ and $RH_\infty$, which may be summarised as follows.
\begin{theorem}\label{thm:statement} Let $X$ be a metric space with a doubling measure and satisfying an annular decay property. Then
\begin{equation}\label{statement} 
A_1(X) = A_\infty\cap e^{BLO(X)},\quad
RH_\infty(X) = e^{BUO(X)}.
\end{equation}
\end{theorem}
Here we have written $w\in e^{BLO}$ whenever $\log w\in BLO$, etc. $BLO$ and $BUO$ stand for functions of \emph{bounded lower} and \emph{upper oscillation}, respectively. By definition, $A_1$ and $RH_\infty$ are the limiting classes of Muckenhoupt $A_p$ weights for $p>1$, and of functions satisfying a reverse Hölder inequality with exponent $s>1$, respectively. While $A_1$ consists of those nonnegative functions $w\in L^1_{\loc}(X)$ such that $Mw\leq Cw$ a.e. in $X$, $RH_\infty$ may be defined by means of the minimal function as the class of those $w$ that satisfy $w\leq Cmw$ a.e. As seen in Lemma \ref{lem:chara}, the maximal and minimal functions can be used to characterize the oscillation spaces $BLO$ and $BUO$. Coupled with the aforementioned commutation result (Lemma \ref{lem:commu}), we find \eqref{statement}. 

Theorem \ref{thm:statement} generalizes Ou's Euclidean result \cite{MR2407089} to metric measure spaces. The characterization of $A_1$ was first discussed by Coifman--Rochberg \cite[Corollary 3]{MR565349}, and the characterization of $RH_\infty$ is due to Cruz-Uribe--Neugebauer \cite[Corollary 4.6]{MR1308005}. All the aforementioned authors are concerned with Euclidean spaces. The complication arising in more general metric measure spaces is that while $A_\infty$ weights do satisfy a reverse Hölder inequality, the reverse is generally not true unless we introduce an additional assumption. To this end, we invoke the so-called \emph{annular decay property,} which will be discussed in more detail in Section \ref{preli}. The principal references to the theory of Muckenhoupt weights on metric measure spaces are \cites{MR0499948,MR1791462,MR1011673}, while \cite{MR807149} is a solid classical treatment in $\mathbb{R}^n$. Cruz-Uribe--Neugebauer's aforementioned article \cite{MR1308005} is a systematic investigation of reverse Hölder classes in the Euclidean context. Some of the theory has been adapted into metric spaces by Kinnunen--Shukla in \cites{MR3265363,MR3130552}. Saito--Tanaka \cite{MR3558538} consider maximal operators and reverse Hölder classes in $\mathbb{R}^n$ with respect to a general base of open sets. Finally, we refer to \cite{MR3310925} for a visual scheme and many more classical references. 

In Section \ref{sec:x3m}, we introduce the natural maximal and minimal functions $M^\natural$ and $m^\natural$ insofar as is necessary to prove Theorem \ref{thm:statement}. In particular, we show that they commute with the logarithm on $A_\infty$, which allows us to pass between $A_p$ and $BMO$ in the proofs that follow. Section \ref{sec:bdd} is dedicated to the following boundedness result for the natural maximal function, generalizing Ou's note \cite{MR1840094} to metric measure spaces. 
\begin{theorem} Let $X$ be a metric space with a doubling measure and satisfying an annular decay property. Then, $M^\natural\mathbin{:}BMO(X) \rightarrow BLO(X)$ is bounded if and only if $M\mathbin{:}A_\infty(X) \rightarrow A_1(X)$ is. 
\end{theorem}

Section \ref{sec:chara} contains the proof of Theorem \ref{thm:statement}, whose general outline was given above. 
In the final section, we demonstrate how a number of known properties of $A_1$ and $RH_\infty$ follow all but immediately from Theorem \ref{thm:statement}. These were already established by Ou \cite{MR2407089} in Euclidean spaces. Provided that the space $X$ satisfies an annular decay property (or, indeed, any sufficient assumption yet to be determined; see the discussion in Section \ref{preli}), the metric-space proofs are the same. This is hardly surprising, because in either case the logarithmic formalism provided by Theorem \ref{thm:statement} reduces arguments into arithmetic with exponents on the oscillation side. The reader may observe that the simplicity of Theorem \ref{thm:statement} is somewhat deceptive, since its level of abstraction requires a degree of regularity of the underlying space.

Notably, the characterization yields a proof of Cruz-Uribe--Neugebauer's refined Jones factorization theorem \cite[Theorem 5.1]{MR1308005}, that includes information on the reverse Hölder classes. 
\begin{theorem}
Let $X$ be a metric space with a doubling measure and satisfying an annular decay property, and $1<p,s<\infty$. The weight $w\in A_p(X)\cap RH_s(X)$ if and only if $w=w_1w_2$ for some $w_1\in A_1(X)\cap RH_s(X)$, $w_2\in A_p(X)\cap RH_\infty(X)$. 
\end{theorem}

\section{Notation and preliminaries}\label{preli}

Throughout, we consider a complete metric measure space $\left(X,\dist,\mu\right)$, where the Borel measure $\mu$ satisfies the doubling condition: there exists a constant $C_d=C_d(\mu) > 1$ such that 
$$
0<\mu\left(2 B\right) \leq C_{d} \mu\left( B \right) < \infty 
$$
for all balls $B\subset X$. The doubling condition implies that $X$ is separable, and since $X$ is complete, it is proper \cite[Proposition 3.1]{MR2867756}. An open ball with center $x\in X$ and radius $r>0$ is denoted $B = B(x, r)$, where the center and radius are left out when not of interest. We occasionally use the notation $cB = B(x, cr)$ for the ball dilated by a constant $c>0$. Various constants are denoted by the letter $C$, whose dependence on parameters is indicated in parentheses when relevant. 

Whenever $E\subset X$ is a measurable subset and the function $f$ is integrable on every compact subset of $E$, we say that $f$ is locally integrable on $E$, denoted $f\in L^1_{\loc}(E)$. If the measure $\nu$ is absolutely continuous with respect to $\mu$ and if there exists a nonnegative locally integrable function $w$ such that $d\nu=w\dd\mu$, we call $\nu$ a weighted measure with respect to $\mu$, and $w$ a weight, following \cite[p.~1]{MR1011673}. For any measurable subset $F\subset E$ and weight $w$ on $E$, we write $w(F)=\int_F w \dd \mu$. The integral average of a function $f\in L^1(E)$ over a measurable set $F\subset E$, with $0<\mu(F)<\infty$, is abbreviated $\mu(F)^{-1}\int_F f\dd\mu = \dashint_F f\dd\mu.$

\begin{dfn} For $1<p<\infty$, a nonnegative function $w\in L^1_{\loc}(X)$ that satisfies 
\begin{equation}\label{dfn:Ap}
\normp{w}{p} = \sup_{B\subset X}\left(\dashint_{B}w\dd\mu \right)\left(\dashint_B w^{-\frac{1}{p-1}}\dd\mu\right)^{p-1} < \infty
\end{equation}
is called a Muckenhoupt $A_p(X)$ weight. When $p=1$, we require instead that there exist a constant $\normp{w}{1}<\infty$ such that for every ball $B\subset X$
\begin{equation}\label{dfn:A1}
\dashint_{B}w\dd\mu \leq \normp{w}{1}\essinf_{x\in B} w(x).
\end{equation}
In both cases, the constant denoted $\normp{w}{p}$ is called the $A_p$ constant of $w$. 
\end{dfn}

As for the class $A_\infty$, we choose the following definition, sometimes called the reverse Jensen inequality. 
\begin{dfn} A nonnegative function $w\in L^1_{\loc}(X)$ belongs to the class $A_\infty(X)$ if 
\begin{equation}\label{revJensen}
\normp{w}{\infty} = \sup_{\substack{B\subset X\\B\ni x}} \left(\dashint_B w\dd\mu\right) \exp\left(-\dashint_B \log w\dd\mu \right) < \infty.
\end{equation}
Again, the constant denoted $\normp{w}{\infty}$ is called the $A_\infty$ constant of $w$. 
\end{dfn}

\begin{dfn} For a $f\in L^1_{\loc}(X)$ the Hardy--Littlewood maximal and Cruz-Uribe--Neugebauer  minimal functions, respectively, are defined by
\begin{equation*}
Mf(x)=\sup_{B\ni x}\dashint_B\abs{f}\dd\mu, \quad mf(x)=\inf_{B\ni x}\dashint_B\abs{f}\dd\mu,
\end{equation*}
where the supremum (infimum) is taken over all balls $B\subset X$ containing the point $x$. 
\end{dfn} 
As is well known, $A_p(X)$ weights are precisely those for which the maximal function is bounded on the weighted space $L^p(X; w\dd\mu)$ \cite[Theorem I.9]{MR1011673}. The minimal function was introduced by Cruz-Uribe and Neugebauer in \cite{MR1308005} to illuminate the structure of the reverse Hölder classes, which will be introduced shortly. It suits our purpose to express the $A_1$ condition \eqref{dfn:A1} in terms of the maximal function:   
\begin{equation}\label{dfn:A1max}
w\in A_1(X) \quad\text{if and only if}\quad Mw(x) \leq \normp{w}{1}w(x)\quad \text{for a.e. }x\in X. 
\end{equation}

\begin{dfn} Let $1<s<\infty$. The nonnegative function $w\in L^1_{\loc}(X)$ belongs to the reverse H\"older class $RH_s(X)$, if there exists a constant $C<\infty$ such that for every ball $B\subset X$
\begin{equation}\label{RHp}
\left(\dashint_B w^s\dd\mu \right)^\frac{1}{s} \leq C\dashint_B w\dd\mu.
\end{equation}
Furthermore, $w$ belongs to the class $RH_\infty(X)$ if there exists a constant $C<\infty$ such that for every ball $B\subset X$
\begin{equation*}
\esssup_{x\in B}w(x)\leq C\dashint_B w\dd\mu
\end{equation*}
or, equivalently, there exists a constant $C<\infty$ such that for almost every $x\in X$
\begin{equation}\label{RHinfty}
w(x)\leq Cmw(x).
\end{equation}
\end{dfn}
The above definitions neatly reflect the symmetrical nature of $A_1$ and $RH_\infty$ as the limiting classes of $A_p$, $1<p<\infty$, and $RH_s$, $1<s<\infty$, as well as the role of the extremal functions with respect to these classes. This will be the subject of Section \ref{sec:chara}. For now, we remark that 
$$A_1 \subsetneq \bigcap_{p>1}A_p,\quad RH_\infty \subsetneq \bigcap_{s>1}RH_s,$$
with a proper inclusion in each case. For discussion and examples see \cite[Lemma 2.3]{MR990859} and \cite[Section 4]{MR1308005}. 

In Euclidean spaces it is well known that the class $A_\infty$ can be characterized in a number of equivalent ways, of which the reverse Hölder \eqref{RHp} and reverse Jensen \eqref{revJensen} inequalities are two. In other words, a weight $w$ belongs to $A_\infty$ if and only if it satisfies a reverse Hölder inequality for some $s>1$. In metric measure spaces, however, this is not universally true. Kinnunen and Shukla \cites{MR3265363,shukla_thesis} give a detailed presentation of the metric-space theory, assuming the so-called $\alpha$-annular decay property. 
\begin{dfn} A metric measure space $(X,d,\mu)$ with a doubling measure $\mu$ is said to satisfy the $\alpha$-annular decay property with $0\leq \alpha\leq 1$, if there exists $C\geq 1$ such that for every $x\in X$, $r>0$, and $0<\delta<1$
\begin{equation}\label{annular}
\mu\left(B(x,r)\setminus B(x, (1-\delta)r)\right) \leq C\delta^\alpha\mu(B(x,r)).
\end{equation}
The constant $C$ is independent of the point, radius, and $\delta$. 
\end{dfn}
Whenever the exact parameter $\alpha$ is not of interest, we say that a space satisfies an annular decay property if it satisfies the $\alpha$-annular decay property for some $0\leq \alpha\leq 1$.

In \cites{MR3265363,shukla_thesis} the annular decay property guarantees that when the underlying measure is doubling, any weighted measure comparable to it is doubling as well. This implication, in turn, is needed in order to pass from qualitative to quantitative characterization of $A_\infty$, and ultimately to the reverse Hölder inequality. It follows that the comparability of measures, reverse Jensen, and reverse Hölder inequalities can be taken as equivalent characterizations of the class $A_\infty$, just as in the Euclidean case. 

Early publications involving the annular decay property or a slight variant are \cites{MR1724375,MR1488297,MR850408}. 
It is well known to hold true for a fairly large class of spaces, including all length spaces; see e.g. \cites{MR1724375,MR3110585}. Strömberg and Torchinsky \cite{MR1011673} develop the Muckenhoupt--reverse Hölder theory in metric measure spaces under the assumption that the measure of a ball depend continuously on its radius. We wish to emphasize that the annular decay property, although sufficient, might not be necessary for our purposes. For convenience, we adopt the assumption \eqref{annular} wherever a reverse Hölder inequality is needed. 

In \cite[Propositions 4.2 and 4.4]{MR2407089}, the logarithmic characterization to be introduced in Section \ref{sec:chara} is used to show the following lemma. In fact, one readily obtains a quantitative result starting from the definitions of the function spaces involved. See also \cite[Lemma 2.2]{MR990859}. 
\begin{lemma}\label{lem:betterthan}Let $X$ be a metric space with a doubling measure and satisfying an annular decay property. 
\begin{enumerate}[label=\normalfont{(\roman*)}]
\item\label{lemome} If $w, w^{-1} \in A_1(X)$, then $w\in L^\infty_{\loc}(X)$;
\item\label{lemtwo} If $w, w^{-1} \in RH_\infty(X)$, then $w\in L^\infty_{\loc}(X)$. 
\end{enumerate} 
\end{lemma}
\begin{proof}\ref{lemome} Let $B\subset X$ be a ball. Using the definition \eqref{dfn:A1} for $w^{-1}$, $$ \dashint_{B}\frac{1}{w}\dd\mu \leq \normp{w}{1}\essinf_{x\in B} \frac{1}{w(x)} = \frac{\normp{w}{1}}{\esssup_B w}.$$
Applying Jensen's inequality and the fact that $w\in A_1$ as well, we conclude that
$$
\esssup_{x\in B} w(x) \leq \normp{w}{1} \left(\dashint_B\frac{1}{w}\dd\mu \right)^{-1} \leq \normp{w}{1}\dashint_B w\dd\mu \leq \normp{w}{1}^2 \essinf_{x\in B} w(x).
$$
This is a Harnack inequality for $w$, guaranteeing that there exist constants $c_1, c_2$ such that $0<c_1\leq w\leq c_2<\infty$ almost everywhere in $B\subset X$ and, in particular, $w\in L^\infty_{\loc}(X)$.

\ref{lemtwo} Under the annular decay assumption it holds that $w, w^{-1}\in A_\infty$, that is, there exist $p,q>1$ such that $w\in A_p$ and $w^{-1}\in A_q$. Without loss of generality we may assume $q=p$. Applying first the $RH_\infty$ \eqref{RHinfty} and then the $A_p$ \eqref{dfn:Ap} condition for $w^{-1}$ we have 
\begin{align*}
\esssup_{x\in B} \frac{1}{w(x)} &\leq C \dashint_B \frac{1}{w}\dd\mu \leq C\normp{w^{-1}}{p}\left(\dashint_B\left(\frac{1}{w}\right)^\frac{1}{1-p}\dd\mu\right)^{1-p}\\
& = C\normp{w^{-1}}{p}\left(\dashint_Bw^\frac{1}{p-1}\dd\mu\right)^{1-p}, 
\end{align*}
which is equivalent to 
\begin{equation*}
\essinf_B w \geq C\normp{w^{-1}}{p}^{-1}\left(\dashint_Bw^\frac{1}{p-1}\dd\mu\right)^{p-1}. 
\end{equation*}
Combining this observation with the $RH_\infty$ and $A_p$ conditions for $w$, and Jensen's inequality, we again obtain 
\begin{align*}
\esssup_B w & \leq C \dashint_B w\dd\mu \leq C\normp{w}{p}\left(\dashint_Bw^\frac{1}{1-p}\dd\mu\right)^{1-p}\\
& \leq C\left(\dashint_Bw^\frac{1}{p-1}\dd\mu\right)^{p-1} \leq C \essinf_B w.
\end{align*}
\end{proof}

Finally, we define the oscillation spaces that we need.  
\begin{dfn}A function $f\in L^1_{\loc}(X)$ is said to be of bounded mean oscillation, denoted $f\in BMO$, when
$$\norm{f}_{BMO} = \sup_{B\subset X}\left(\dashint_B \abs{f - f_B}\dd\mu \right) < \infty.$$ 
Functions of bounded upper and lower oscillation ($BUO$ and $BLO$), respectively, are defined as those functions $f\in L^1_{\loc}(X)$ for which
\begin{gather*}
\norm{f}_{BUO} = \sup_{B\subset X}\left(\esssup_{x\in B} f(x) - \dashint_B f\dd\mu \right) < \infty,\\ 
\norm{f}_{BLO} = \sup_{B\subset X}\left(\dashint_B f\dd\mu - \essinf_{x\in B} f(x)\right) <\infty.
\end{gather*}
\end{dfn}
In spite of the notation, none of the above functionals are norms: consider $\norm{\cdot}_{BMO}$ of almost everywhere constant functions, and the fact that $\norm{-f}_{BUO} = \norm{f}_{BLO}$. We also remark that $BUO$ and $BLO$ are subsets of $BMO$.

\section{The natural extremal functions}\label{sec:x3m}

The natural maximal and minimal functions were studied by Ou in \cites{MR1840094,MR2407089}. In this section we introduce these functions and show that they commute with the logarithm on $A_\infty$. Furthermore, we show how the natural extremal functions may be used to characterize the oscillation spaces $BUO$ and $BLO$. As simple as these observations are, they will unlock an elegant formalism in Section \ref{sec:chara}.

\begin{dfn}For a function $f\in L^1_{\loc}(X)$ its natural maximal and minimal functions, respectively, are given by
\begin{equation*}
M^\natural f(x) = \sup_{B\ni x}\dashint_B f\dd\mu, \quad
m^\natural f(x) = \inf_{B\ni x}\dashint_B f\dd\mu,
\end{equation*}
where the supremum (infimum) is taken over all balls $B\subset X$ containing the point $x$. 
\end{dfn}
As a matter of fact we have $M^\natural f(x) = -m^\natural(-f)(x)$. 

\begin{lemma}[\cite{MR1840094}]\label{lem:commu} Let $X$ be a metric space with a doubling measure, and $w\in A_\infty(X)$. We have
\begin{gather*}
0 \leq \left(\log M^\natural - M^\natural \log\right)w \leq \log \normp{w}{\infty}, \\
0 \leq \left(\log m^\natural - m^\natural \log\right)w \leq \log \normp{w}{\infty}.
\end{gather*}
\end{lemma}
\begin{proof} We show the statement for $M^\natural$. By the fact that $w\in A_\infty$ \eqref{revJensen} and Jensen's inequality,
\begin{equation}\label{minf}
\dashint_B w\dd\mu \leq \normp{w}{\infty}\exp\left(\dashint_B \log w\dd\mu \right) \leq \normp{w}{\infty}\dashint_B w \dd\mu.
\end{equation}
Taking $\sup$, $\log$, and rearranging the terms, 
\begin{gather*}
\log M^\natural w \leq \log\normp{w}{\infty} + M^\natural (\log w) \leq \log\normp{w}{\infty} + \log M^\natural w \\
0 \leq \log\normp{w}{\infty} + \left(M^\natural \log - \log M^\natural \right) w \leq \log\normp{w}{\infty} + \log M^\natural w \\ 
0 \leq \left(\log M^\natural - M^\natural \log \right) w \leq \log\normp{w}{\infty}. 
\end{gather*}
The proof for $m^\natural$ is analogous, taking the infimum of \eqref{minf} instead.
\end{proof}

\begin{lemma}[\cite{MR660603}, Lemma 2]\label{lem:chara} Let $X$ be a metric space with a doubling measure, and $f\in L^1_{\loc}(X)$. Then $f\in BLO(X)$ if and only if $M^\natural f-f\in L^\infty(X)$, and $f\in BUO(X)$ if and only if $f - m^\natural f \in L^\infty(X)$, with 
\begin{equation*}
\norm{f}_{BLO} = \norm{M^\natural f-f}_\infty, \quad 
\norm{f}_{BUO} = \norm{f - m^\natural f}_\infty.
\end{equation*}
\end{lemma}
\begin{proof} We show the statement for $BLO$; the proof for $BUO$ is the same, making the necessary modifications. Assume first that $f\in BLO$, let $x$ be a Lebesgue point of $f$, and $B\subset X$ a ball with $B \ni x$. Then 
$$ \dashint_B f\dd\mu - f(x) \leq \dashint_B f\dd\mu -\essinf_{\substack{B\ni x\\B\subset X}} f(x) \leq \norm{f}_{BLO}.
$$
Taking the supremum over all balls $B\ni x$ we obtain
$$ M^\natural f(x) - f(x) \leq \norm{f}_{BLO}.
$$
Since $x\in X$ is an arbitrary Lebesgue point of $f$, the above implies 
$$ \norm{M^\natural f - f}_\infty \leq \norm{f}_{BLO}. $$

Conversely, let $M^\natural f-f\in L^\infty$. If $B\subset X$ is any ball, $x\in B$, and 
\begin{equation}\label{impo}
f(x) < \dashint_B f \dd\mu - \norm{M^\natural f-f}_\infty, 
\end{equation}
then we have
$$ M^\natural f(x) - f(x) \geq \dashint_B f \dd\mu - f(x) > \norm{M^\natural f-f}_\infty,$$
which means that the set of points $x$ satisfying \eqref{impo} is of measure zero. This implies 
\begin{align*}
\dashint_B f \dd\mu - \norm{M^\natural f-f}_\infty &\leq \essinf_{\substack{B\ni x\\B\subset X}} f(x). 
\end{align*} 
Since the ball $B$ is arbitrary, this is equivalent to $\norm{f}_{BLO} \leq \norm{M^\natural f - f}_\infty$. 
\end{proof}

\section{Boundedness of the natural maximal function}\label{sec:bdd}

We show that the natural maximal function is bounded from $BMO$ to $BLO$ whenever the Hardy--Littlewood maximal function from $A_\infty$ to $A_1$ is. This fact is not directly in our main line of investigation, but it generalizes the Euclidean result of \cite{MR1840094}. Note that we assume an annular decay property to ensure that $w\in A_\infty$ implies a reverse Hölder inequality for $w$.

The following preliminary result can be thought of as a weaker version of the Coifman--Rochberg lemma. The strategy of the proof and treatment of the integral outside $4B$ are in fact similar to the proof of the Coifman--Rochberg lemma; compare \cite[Proposition 2.10]{kurki_mudarra_ext}.
\begin{lemma}\label{lem:mainfty}
Let $X$ be a metric space with a doubling measure, and $w\in L^1_{\loc}(X)$ a nonnegative function. If $w\in RH_s(X)$, then $Mw\in RH_s(X)$. Furthermore, whenever $X$ satisfies an annular decay property, this implies $M(A_\infty(X))\subset A_\infty(X)$.
\end{lemma}
\begin{proof} Assuming that the space $X$ satisfies an annular decay property, $w\in A_\infty(X)$ is equivalent to $w\in RH_s(X)$ for some $s>1$ \cite[Theorem 3.1]{MR3265363}. We are going to show the statement on the reverse Hölder side.

From here on, let $q>1$ denote the reverse Hölder class of $w$. Let $B$ be a ball in $X$, and write $w = w\Chi_{4B} - w(1-\Chi_{4B}) = w_1 + w_2$. Owing to the subadditivity of the maximal function it holds that $Mw \leq Mw_1 + Mw_2$, and
$$
\left(\dashint_B \left(Mw\right)^q\dd\mu\right)^\frac{1}{q} \leq 
2^\frac{q-1}{q}\left( \left(\dashint_B \left(Mw_1\right)^q\dd\mu\right)^\frac{1}{q} + \left(\dashint_B \left(Mw_2\right)^q\dd\mu\right)^\frac{1}{q} \right).
$$
We estimate each term separately. Using respectively the $L^q(X)$ bound for the maximal function, the doubling property of $\mu$, and the reverse Hölder inequality \eqref{RHp} for $w$, we obtain
\begin{align*}
\left(\dashint_B \left(Mw_1\right)^q\dd\mu\right)^\frac{1}{q} & \leq C\left(\frac{1}{\mu(B)}\int_X w_1^q\dd\mu\right)^\frac{1}{q} = C\left(\frac{1}{\mu(B)}\dashint_{4B} w^q\dd\mu\right)^\frac{1}{q}\\
& \leq C(C_d)\left(\dashint_{4B}w^q\dd\mu\right)^\frac{1}{q} \leq C \dashint_{4B}w\dd\mu. 
\end{align*}
Under the annular decay assumption, $w\in A_\infty$ implies that the weighted measure $w\dd\mu$ is comparable to the underlying measure $\mu$, and doubling whenever $\mu$ is \cite{MR3265363}. The doubling property for $w\dd\mu$ lets us conclude that  
$$
\left(\dashint_B \left(Mw_1\right)^q\dd\mu\right)^\frac{1}{q} \leq C\dashint_{4B}w\dd\mu \leq C(C_d) \dashint_B w\dd\mu \leq C \dashint_B Mw\dd\mu.
$$

Next we estimate $Mw_2$. Let $x,y\in B=B(z,r)\subset X$ and let $B'=B(z',r')$ be another ball in $X$ containing $y.$ Assume first that there exists a point $p\in B' \setminus 4B.$ We claim that $r \leq r'.$ Indeed, otherwise we have $\dist(y,p) \leq 2r'\leq 2r$ and 
$$
\dist(y,z) \geq \dist(z,p)-\dist(p,y) \geq 4 r- 2r=2r> r,
$$
implying that $y\notin B,$ a contradiction. Using the fact that $r \leq r'$ we have for any $q \in B$
$$
\dist(q,z') \leq \dist(q,y) + \dist(y,z') \leq 2r'+r'=3r',
$$
which shows that $B \subset 4 B'.$ In particular $x\in 4B'$ and we can write
\begin{equation*}
\dashint_{B'}w_2\dd\mu \leq C(C_d)\dashint_{4B'}w_2\dd\mu \leq C\sup_{\substack{B\ni x\\B\subset X}}  \dashint_{B}w\dd\mu = C (Mw)(x). 
\end{equation*}
If $B' \subset 4B$ then $\int_{B'}w_2=0,$ and the preceding estimate trivially holds. Either way the right-hand side is independent of $y$, and thus
\begin{equation*}
Mw_2(y) = \sup_{\substack{B\ni y\\B\subset X}}\dashint_{B}w\dd\mu \leq C (Mw)(x)
\end{equation*}
for any and all points $x,y\in B$. Raising to power $q$, and integrating first with respect to $y$ and then to $x$, we find
$$
\left(\dashint_B \left(Mw_2\right)^q\dd\mu\right)^\frac{1}{q} \leq C \dashint_B Mw\dd\mu.
$$ 
\end{proof} 

\begin{lemma}\label{lem:mmapsto} Let $X$ be a metric space with a doubling measure and satisfying an annular decay property, and $w$ a weight. If $w\in A_\infty(X)$, then $Mw\in A_1(X)$, with $\normp{Mw}{1}$ only depending on the doubling constant $C_d$. 
\end{lemma}
\begin{proof}
Assuming that $X$ satisfies an annular decay property, for a weight $w\in A_\infty$ there exists $s>1$ such that $w$ satisfies the reverse Hölder inequality \eqref{RHp} \cite[Theorem 3.1]{MR3265363}, implying
\begin{equation*}
C\left(\dashint_B w^s \dd\mu \right)^\frac{1}{s} \leq \dashint_B w \dd\mu \leq \left(\dashint_B w^s \dd\mu \right)^\frac{1}{s}
\end{equation*}
or, in terms of the maximal function, 
\begin{equation}\label{xximals}
C\left(Mw^s\right)^\frac{1}{s} \leq Mw \leq \left(Mw^s\right)^\frac{1}{s}.
\end{equation}
By Lemma \ref{lem:mainfty} $Mw$ belongs to $RH_s$. Applying the reverse Hölder inequality for $Mw$, we find that $\left(Mw\right)^s\in A_\infty$ with $ \normp{\left(Mw\right)^s}{\infty} = C^s\normp{Mw}{\infty}^s. $
Now by \eqref{xximals} the function $Mw^s$ is comparable to $ \left(Mw\right)^s$, so it also lies in $A_\infty$ with 
 $$\normp{Mw^s}{\infty} = C\normp{Mw}{\infty}.$$
Here the constant $C$ only depends on the doubling constant $C_d$ and parameters $C, s$ of the reverse Hölder inequality for $w$. The latter, in turn, are again controlled by $C_d(\mu)$; see Theorem 2.15 and the preceding discussion in \cite{shukla_thesis}. The Coifman--Rochberg lemma, whose metric-space proof can be found in \cite[Proposition 2.10]{kurki_mudarra_ext}, implies that $\left(Mw^s\right)^{1/s}$, $s>1$, belongs to $A_1$. Precisely, 
\begin{equation}\label{A1mws}
M\left(\left(Mw^s\right)^\frac{1}{s}\right) \leq C(C_d,s)\left(Mw^s\right)^\frac{1}{s}.
\end{equation}
Applying \eqref{xximals} and the $A_1$ condition \eqref{A1mws} for $\left(Mw^s\right)^{1/s}$, we obtain
$$
M\left(Mw\right) \leq M\left(\left(Mw^s\right)^\frac{1}{s}\right) \leq C(C_d, s)\left(Mw^s\right)^\frac{1}{s} \leq CMw,
$$
whereby the statement is proven.
\end{proof}

\begin{lemma}\label{lem:logmaps} Let $X$ be a metric space with a doubling measure, and $w$ a weight. If $w\in A_1(X)$, then $\log w\in BLO(X)$.
\end{lemma}
\begin{proof}
Let $B\subset X$ be a ball, apply Jensen's inequality and the $A_1$ condition \eqref{dfn:A1} for $w$:
\begin{align*}
\dashint_B\log w\dd\mu \leq \log \dashint_B w\dd\mu \leq \log\left(\normp{w}{1}\essinf_{x\in B} w(x)\right) = \log \normp{w}{1} + \essinf_{x\in B} \log w(x).
\end{align*}
Rearranging, we obtain 
$$ 
\dashint_B\log w\dd\mu - \essinf_{x\in B}\log w(x) \leq \log \normp{w}{1}
$$
for every $B\subset X$, which is the $BLO$ condition for $\log w$. 
\end{proof}

\begin{propo}\label{propo:bounded} Let $X$ be a metric space with a doubling measure and satisfying an annular decay property. If $f\in BMO(X)$, then $M^\natural f\in BLO(X)$ with 
$$\lVert M^\natural f\rVert_{BLO}\leq C(C_d)\norm{f}_{BMO}.$$
\end{propo}
\begin{proof}
Let $f\in BMO$. It follows from the John--Nirenberg lemma that $f = a \log w$ for some $a = a(C_d) > 0$ and $w\in A_2$. 
For this lemma and corollary in metric spaces see \cite[Sections 3.3 and 3.4]{MR2867756} and \cite[Appendix]{MR1756109}, as well as \cite[Theorem 3.11]{shukla_thesis}. By the commutation Lemma \ref{lem:commu} we have
$$
M^\natural f(x) = a\left(\log Mw(x)+b(x)\right),
$$
where $b\in L^\infty$. As for the $\log Mw$ term, by Lemma \ref{lem:mmapsto} $M$ maps $A_\infty$ into $A_1$, and by Lemma \ref{lem:logmaps} $\log$ maps $A_1$ into $BLO$, with norms and characteristic constants depending only on the underlying measure $\mu$. 
\end{proof}
\begin{coro}\label{coro:bounded}
Let $X$ be a metric space with a doubling measure and satisfying an annular decay property. If $f\in BMO(X)$, then $M f\in BLO(X)$ with $$\norm{M f}_{BLO}\leq C(C_d)\norm{f}_{BMO}.$$
\end{coro}
\begin{proof}
Let $f\in BMO$. By Proposition \ref{propo:bounded} $Mf = M^\natural(\abs{f})\in BLO$, and 
$$ \norm{Mf}_{BLO} = \lVert M^\natural(\abs{f})\rVert_{BLO} \leq 2C(C_d)\norm{f}_{BMO}.
$$
\end{proof}
We note that Proposition \ref{propo:bounded} and Corollary \ref{coro:bounded} can be adapted for the natural minimal function by the algebraic fact that $M^\natural f = -m^\natural\left(-f\right)$. 
\begin{coro}\label{coro:minimal}
Let $X$ be a metric space with a doubling measure and satisfying an annular decay property. If $f\in BMO(X)$, then  
\begin{gather*}
m^\natural f\in BUO(X) \quad \text{with}\quad \lVert m^\natural f\rVert_{BUO} \leq C(C_d)\norm{f}_{BMO},\\
m f\in BUO(X) \quad \text{with}\quad \norm{mf}_{BUO}\leq C(C_d)\norm{f}_{BMO}.
\end{gather*}
\end{coro}
\begin{proof}
Let $f\in BMO$ and denote $y = -f$. Now $y\in BMO$ as well, and 
$$m^\natural f = m^\natural(-y) = - M^\natural y = -M^\natural(-f). 
$$
By Proposition \ref{propo:bounded} $M^\natural(-f)$ is in $BLO$, so $-M^\natural(-f) = m^\natural f\in BUO$ with 
\begin{align*}
\lVert m^\natural f\rVert_{BUO} & = \lVert -M^\natural (-f)\rVert _{BUO} = \lVert M^\natural (-f)\rVert_{BLO}\\
& \leq C(C_d)\norm{-f}_{BMO} = C\norm{f}_{BMO}.
\end{align*}
The second statement follows immediately, since now $mf = m^\natural(\abs{f})\in BUO$ and
$$ \norm{mf}_{BUO} = \lVert m^\natural(\abs{f})\rVert_{BUO} \leq 2C(C_d)\norm{f}_{BMO}. $$
\end{proof}

The following Proposition reproduces the final theorem of \cite{MR1840094}. This can be read as the converse of Proposition \ref{propo:bounded}, implying that in spaces satisfying an annular decay property, the boundedness of $M^\natural\mathbin{:}BMO \rightarrow BLO$ is equivalent to $M(A_\infty)\subset A_1$. In either case the John--Nirenberg lemma is instrumental in making the connection with $A_p$ and $BMO$.

\begin{propo} Let $X$ be a metric space with a doubling measure and satisfying an annular decay property. Assume that $M^\natural\mathbin{:}BMO(X) \rightarrow BLO(X)$ with $\norm{M^\natural f}_{BLO}\leq C(C_d)\norm{f}_{BMO}$. Then, if $w\in A_\infty(X)$, we have $Mw\in A_1(X)$ with $\normp{Mw}{1}$ only depending on the doubling constant $C_d$. 
\end{propo}
\begin{proof}
Let $w\in A_\infty$. By the John--Nirenberg lemma and the assumption we have $\log w\in BMO$ and $M^\natural\log w \in BLO$. The commutation Lemma \ref{lem:commu} gives 
\begin{equation}\label{firstCommu}
M^\natural(\log w) \leq \log Mw \leq \log \normp{w}{\infty} + M^\natural(\log w). 
\end{equation}
By Lemma \ref{lem:mainfty} $Mw$ also lies in $A_\infty$, and we can apply the commutation lemma to it: 
\begin{gather}
\nonumber M^\natural(\log Mw) \leq \log M(Mw) \leq \log\normp{Mw}{\infty} + M^\natural(\log Mw)\\
\label{fcondCommu} e^{M^\natural(\log Mw)} \leq M(Mw) \leq \normp{Mw}{\infty}e^{M^\natural(\log Mw)}.
\end{gather}
Combining \eqref{firstCommu} and \eqref{fcondCommu} leads to 
\begin{equation}
\label{frd} e^{M^\natural(M^\natural\log w)} \leq M(Mw) \leq \normp{Mw}{\infty}\normp{w}{\infty}e^{M^\natural(M^\natural\log w)}.
\end{equation}
Recall that by Lemma \ref{lem:chara}
\begin{equation}\label{froth}
M^\natural f - f \leq \norm{f}_{BLO} <\infty.
\end{equation}
Applying \eqref{frd}, \eqref{froth}, and \eqref{firstCommu} in this order, we obtain
\begin{align*}
M(Mw) & \leq \normp{Mw}{\infty}\normp{w}{\infty}e^{M^\natural(M^\natural\log w)}\\
& \leq \normp{Mw}{\infty}\normp{w}{\infty}e^{M^\natural(\log w) + \norm{M^\natural\log w}_{BLO}}\\
& \leq \normp{Mw}{\infty}\normp{w}{\infty}e^{\log Mw + \norm{M^\natural\log w}_{BLO}}\\
& = \normp{Mw}{\infty}\normp{w}{\infty}e^{\lVert M^\natural\log w\rVert_{BLO}}Mw.
\end{align*}
By the assumption, we conclude that $Mw\in A_1$ with $$\normp{Mw}{1} \leq \normp{Mw}{\infty}\normp{w}{\infty}C(C_d)e^{\norm{\log w}_{BMO}}.$$
\end{proof}

\section{Two characterizations}\label{sec:chara}

This section takes us back to the characterization of the spaces $A_1$ and $RH_\infty$. Theorems \ref{thm:Aside} and \ref{thm:RHside} bring \cite[Theorem 3.1 and 3.2]{MR2407089}, respectively, into the metric setting. The proofs are virtually the same as in the Euclidean case. On the reverse Hölder side (Theorem \ref{thm:RHside}), we require the annular decay property in order to access the $A_\infty$ condition. On the contrary, the characterization of $A_1$ (Theorem \ref{thm:Aside}) can be generalized as is. 
\begin{theorem} \label{thm:Aside}
Let $X$ be a metric space with a doubling measure, and $w$ a weight. Then $w\in A_1(X)$ if and only if $w\in A_\infty(X)$ and $\log w\in BLO(X)$. Furthermore,
$$ \exp\left(\norm{\log w}_{BLO}\right) \leq \normp{w}{1} \leq \normp{w}{\infty}\exp\left(\norm{\log w}_{BLO}\right).$$
\end{theorem}
\begin{proof}
Assume first that $w\in A_1$. This direction is well known, the proof in $\mathbb{R}^n$ being the same; see e.g. \cite[Theorem II.3.3]{MR807149}. We denote $\log w = \varphi$. Since $w\in A_1$, it also belongs to $A_\infty$. Then, by definition, for almost every $x\in B\subset X$ we have $\dashint_B \exp\varphi \leq \normp{w}{1}\exp\varphi$. In other words,
\begin{gather*}
\esssup_{x\in B} e^{-\varphi(x)}\dashint_B e^{\varphi} \dd\mu \leq \normp{w}{1}\\
e^{-\essinf_B\varphi(x)}\dashint_B e^{\varphi} \dd\mu \leq \normp{w}{1}.
\end{gather*}
Jensen's inequality gives $\exp\left(\dashint_B \varphi\right) \leq \dashint_B \exp(\varphi)$, so
$$\exp\left(-\essinf_{x\in B}\varphi(x)+\dashint_B\varphi \right) \leq \normp{w}{1},
$$
which means that $\dashint_B\varphi - \essinf_B\varphi \leq \log\normp{w}{1}$, that is, $\log w\in BLO$ with 
$\norm{\log w}_{BLO} \leq \log\normp{w}{1}.$

Conversely, assume that $w\in A_\infty$ and $\log w\in BLO$. Lemmas \ref{lem:chara} and \ref{lem:commu}, respectively, state that almost everywhere
\begin{gather*}
M^\natural\log w \leq \log w + \norm{\log w}_{BLO},\text{ and}\\
\log M^\natural w - \log\normp{w}{1} \leq M^\natural\log w.
\end{gather*}
Combining these gives for almost every $x\in X$ the inequality 
\begin{gather*}
\log M^\natural w  \leq \log w + \norm{\log w}_{BLO} + \log\normp{w}{\infty}\\
Mw(x) \leq  \left(\normp{w}{\infty}e^{\norm{\log w}_{BLO}}\right)w(x),
\end{gather*}
implying that $w\in A_1$ with
$ \normp{w}{1} \leq \normp{w}{\infty}\exp\left(\norm{\log w}_{BLO}\right).$ 
Moreover, the case of constant weights shows the bound to be sharp. 
\end{proof}

\begin{theorem}\label{thm:RHside} 
Let $X$ be a metric space with a doubling measure and satisfying an annular decay property, and $w$ a weight. Then 
$w\in A_\infty(X)\cap RH_\infty(X)$ if and only if $\log w\in BUO(X)$ with 
$$ C \leq \exp(\norm{\log w}_{BUO}) \leq C\normp{w}{\infty},$$ 
where $C$ is the infimal constant in the $RH_\infty$ condition \eqref{RHinfty} for $w$.
\end{theorem}
\begin{remark}Assuming an annular decay property, it in fact holds that $A_\infty\cap RH_\infty = RH_\infty$; see the proof immediately below. We keep the redundant statement in order to illustrate the significance of the extra assumption.
\end{remark}
\begin{proof}[Proof of Theorem \ref{thm:RHside} ]
Assume first that $\log w\in BUO$. The characterization of $BUO$ (Lemma \ref{lem:chara}) gives
$$ m^\natural \log w \geq \log w - \norm{\log w}_{BUO} \quad\text{a.e. }x\in X, $$
while $\log m^\natural w \geq m^\natural \log w $ almost everywhere by Jensen's inequality. Hence
\begin{gather*}
\log m^\natural w \geq \log w - \norm{\log w}_{BUO}\\
mw(x) \geq \exp\left(\norm{\log w}_{BUO}\right)w(x)
\end{gather*}
for almost every $x\in X$, that is, $w\in RH_\infty$ with constant $C \leq \exp(\norm{\log w}_{BUO})$, as desired. This also implies that $w\in A_\infty$ because
\begin{equation}\label{RHin}
RH_\infty \subsetneq \bigcap_{s>1} RH_s \subset \bigcup_{s>1} RH_s = A_\infty.
\end{equation}
The first inclusion is \cite[Theorem 4.1]{MR1308005} (in $\mathbb{R}^n$, but the proof in metric spaces is the same.) The final equality is valid under the annular decay assumption; see \cite{MR3265363} and the discussion in Section \ref{preli}.

Say, then, that $w\in A_\infty\cap RH_\infty$. Combining the definitions \eqref{revJensen} and \eqref{RHinfty}, we obtain for almost every $x$ in an arbitrary ball $B\subset X$
\begin{gather*}
\normp{w}{\infty} \exp\left(\dashint_B\log w\dd\mu\right) \geq \dashint_B w \dd\mu \geq C^{-1}w\\
\log\normp{w}{\infty} + \dashint_B\log w \dd\mu \geq \log w -\log C\\
\log\normp{w}{\infty} + \log C \geq \log w - \dashint_B\log w \dd\mu.
\end{gather*}
Taking the essential supremum over points $x\in B$, we have
$$ \esssup_B\log w - \dashint_B\log w \leq \log\left(C\normp{w}{\infty}\right),$$
which is the $BUO$ condition for $\log w$ with $\norm
{\log w}_{BUO}\leq \log\left(C\normp{w}{\infty}\right)$.
\end{proof}

\section{Applications to Muckenhoupt weights}\label{sec:appli}

Using the results of the previous section, various properties of the spaces $A_1$ and $RH_\infty$ follow all but immediately. This section adapts most of \cite[Section 4]{MR2407089} into the metric setting, culminating in the refined Jones factorization Theorem \ref{thm:jones}. Propositions 4.2 and 4.4 in \cite{MR2407089} have been omitted in favour of the direct and quantitative Lemma \ref{lem:betterthan} above. The reader may note that the annular decay assumption appears where needed.

The following lemma (on $\mathbb{R}$) was shown in \cite[Theorem 2.2]{MR990859}. 
\begin{propo}
Let $X$ be a metric space with a doubling measure, and $w$ a weight. If $w\in A_\infty(X)$ and there exists $s>0$ such that $w^s\in A_1(X)$, then $w\in A_1(X)$. 
\end{propo}
\begin{proof}
By Theorem \ref{thm:Aside}, $w^s\in A_1$ implies that $\log w^s = s\log w\in BLO$. Since $BLO$ is closed under multiplication by positive scalars, $\log w\in BLO$ as well.
\end{proof}

The next proposition is the metric-space version of \cite[Theorem 4.2]{MR1308005}.
\begin{propo} Let $X$ be a metric space with a doubling measure and satisfying an annular decay property, and $w$ a weight. The following statements are equivalent. 
\begin{enumerate}[label=\normalfont{(\roman*)}]
\item\label{i} $w\in RH_\infty(X)$;
\item\label{ii} $w^{s_0}\in RH_\infty(X)$ for some $s_0>0$;
\item\label{iii} $w^s\in RH_\infty(X)$ for all $s>0$.
\end{enumerate}
\end{propo}
\begin{proof}
The implication \ref{ii} $\Rightarrow$ \ref{iii} follows from the characterization in Theorem \ref{thm:RHside}. Recall that under the annular decay assumption $RH_\infty \subset A_\infty$, so $w\in A_\infty$. Now $w^{s_0}\in RH_\infty$ means that 
$$\log w^{s_0} = s_0 \log w\in BUO.$$ 
Since $BUO$ is closed under multiplication by positive scalars, \ref{iii} follows. The other implications are immediate because we may choose $s_0=1$. 
\end{proof}

A function $\varphi$ such that $\varphi w\in RH_\infty$ for all $w\in RH_\infty$ is called a \emph{multiplier} of $RH_\infty$. 
\begin{propo}\label{propo:multinfty} Let $X$ be a metric space with a doubling measure and satisfying an annular decay property. 
A function $\varphi\in A_\infty(X)$ is a multiplier of $RH_\infty(X)$ if and only if $\varphi\in RH_\infty(X)$.  
\end{propo}
\begin{proof}
If $\varphi$ is a multiplier of $RH_\infty$, then $\varphi = 1\cdot\varphi$ belongs to $RH_\infty$ because $w\equiv 1$ does. Assume now that $\varphi\in A_\infty\cap RH_\infty$, and let $w$ be a function in $RH_\infty$. Recall that according to \eqref{RHin} we have $w\in A_\infty$. Consequently, Theorem \ref{thm:RHside} allows us to write 
\begin{gather*}
\varphi = \exp\left(\log\varphi\right) = \exp f,\\
w = \exp\left(\log w\right) = \exp g,
\end{gather*}
where $f, g\in BUO$. Because $BUO$ is closed under vector addition, we have $w\varphi = \exp(f+g)\in BUO$.
\end{proof}

In \cite[Theorem 4.10]{MR1308005}, Cruz-Uribe and Neugebauer characterize the set of multipliers of $A_1$ as those nonnegative, locally integrable functions $\varphi$ for which 
$$\frac{1}{\varphi}\in RH_\infty\cap \bigcap_{p>1}A_p.$$
The proof relies on the following result, which is their Theorem 4.4. The multipliers of $A_1$ have previously been discussed in \cite[Theorem 2.14]{MR990859}; see also \cite{MR3130552}.
\begin{propo}\label{propo:mult} Let $X$ be a metric space with a doubling measure and satisfying an annular decay property, and $p>1$. A weight $w\in A_1(X)$ if and only if $w^{1-p}\in A_p(X)\cap RH_\infty(X)$. 
\end{propo}
\begin{proof}
If $w\in A_1 \subsetneq \bigcap_{p>1}A_p$ and $p'$ is the conjugate exponent given by 
$$\frac{1}{p}+\frac{1}{p'} = 1,$$ 
then $w\in A_{p'}$ and $w^{1-p}\in A_p$. Furthermore, by Theorem \ref{thm:Aside} we have $\log w\in BLO$, which implies 
$$\log w^{1-p} = -(p-1)\log w\in BUO.$$ 
According to the characterization Theorem \ref{thm:RHside}, this means that $w^{1-p}\in RH_\infty$.

Assume then that $w^{1-p}\in A_p\cap RH_\infty$. The $A_p$ condition implies that 
$$ w^{(1-p)(1-p')} = w\in A_{p'}\subset A_\infty.$$ 
In addition, since $w^{1-p}\in RH_\infty$, Theorem \ref{thm:RHside} states that $\log w^{1-p} \in BUO$, which in turn implies that 
$$ (1-p')\log w^{(1-p)} = \log w^{(1-p)(1-p')} = w\in BLO, $$
because $1-p'<0$. Altogether we find that $w\in A_\infty$ and $\log w\in BLO$, which is equivalent to $w\in A_1$ by Theorem \ref{thm:Aside}. 
\end{proof}

Finally, we apply Propositions \ref{propo:mult} -- \ref{propo:s} to show a refined Jones factorization theorem \cite[Theorem 5.1]{MR1308005} that includes information on reverse Hölder classes. With the preliminary results in place, the proof is identical to \cite[Theorem 4.9]{MR2407089}. 

For reference, we state the following lemma from \cite{MR1018575}, where it is labelled preliminary result {P6}. 
The proof is a straightforward application of the definitions of the function spaces involved.
\begin{propo}\label{propo:s}
Let $1<p,s<\infty$. The weight $w\in A_p(X)\cap RH_s(X)$ if and only if $w^s\in A_{s(p-1)+1}(X).$
\end{propo}
Furthermore, we are going to need the classical Jones factorization theorem in metric spaces. For a proof see \cite[Proposition 2.9]{kurki_mudarra_ext} and let $E=X$; this nonstandard version factorizes a weight $w$ such that $w^r\in A_p$ with $r>1$, but letting $r=1$ only simplifies the proof. 
\begin{theorem}\label{thm:jones}
Let $X$ be a metric space with a doubling measure and satisfying an annular decay property, and $1<p,s<\infty$. The weight $w\in A_p(X)\cap RH_s(X)$ if and only if 
$$w=w_1w_2$$ 
for some $w_1\in A_1(X)\cap RH_s(X)$, $w_2\in A_p(X)\cap RH_\infty(X)$. 
\end{theorem}
\begin{remark}The cases $p=1$ or $s=\infty$ are immediate, the latter by Proposition \ref{propo:multinfty}. 
\end{remark}
\begin{proof}[Proof of Theorem \ref{thm:jones}]
By Proposition \ref{propo:s}, $w\in A_1$ if and only if $w^s\in A_{s(p-1)+1}$. By the original Jones factorization, the statement for $w^s$ is true if and only if there exist $v_1, v_2\in A_1$ such that $w^s = v_1v_2^{-s(p-1)}$, that is, $w = v_1^{1/s}v_2^{1-p}$. Proposition \ref{propo:s} for $p=1$ states that $v_1\in A_1$ is equivalent to 
$$v_1^{1/s}\in A_1\cap RH_s.$$ 
Furthermore, by Proposition \ref{propo:mult}, $v_2\in A_1$ is equivalent to 
$$v_2^{1-p}\in A_p\cap RH_\infty.$$ 
Choosing $w_1 = v_1^{1/s}$ and $w_2 = v_2^{1-p}$, the statement is proven. 
\end{proof}

\section*{Statements and Declarations}
\subsection*{Funding}
The author was supported by the Vilho, Yrjö and Kalle Väisälä Foundation of the Finnish Academy of Science and Letters.

\subsection*{Conflict of interests}
The author has no relevant financial or non-financial interests to disclose.

\subsection*{Data availability}
Data sharing not applicable to this article as no datasets were generated or analysed during the current study.

\end{document}